\documentclass[a4paper,12pt,reqno]{amsart}

\usepackage[utf8]{inputenc}
\usepackage{latexsym}
\usepackage{amssymb}
\usepackage{amsmath}
\usepackage{enumerate}
\usepackage{xcolor}

\newtheorem{theorem}{Theorem}[section]
\newtheorem{lemma}[theorem]{Lemma}
\newtheorem{proposition}[theorem]{Proposition}
\newtheorem{corollary}[theorem]{Corollary}
\newtheorem{example}[theorem]{Example}

\newtheorem{remark}[theorem]{Remark}
\newtheorem{definition}[theorem]{Definition}

\def\N{\mathbb N}

\def\C{\mathbb C}

\newcommand\vN{\mathop{\rm VN}}
\newcommand\supp{\mathop{\rm supp}}

\DeclareMathOperator{\spnn}{span}
\newcommand\Ag{\mathop{A(G/G_e)}}
\newcommand\Ah{\mathop{A(H/H_e)}}
\newcommand\Ge{\mathop{G/G_e}}
\newcommand\He{\mathop{H/H_e}}
\newcommand\la{\lambda}
\newcommand\Omo{\Omega_{\text{o}}}
\newcommand\Omoc{\Omega_{\text{o}}^{\text{c}}}

\title{Linear preservers on idempotents of Fourier algebras}

\author{Ying-Fen Lin and Shiho Oi}

\address[Y.-F. Lin]{Mathematical Sciences Research Centre, Queen's University Belfast, Belfast, BT7 1NN, U.K. \\
Email: y.lin@qub.ac.uk}
\address[S. Oi]{Department of Mathematics, Faculty of Science, Niigata University, Niigata 950-2181, Japan \\
Email: shiho-oi@math.sc.niigata-u.ac.jp}


\begin{document}

\maketitle

\begin{abstract}
In this article, we give a representation of bounded complex linear operators which preserve idempotent elements on the Fourier algebra of a locally compact group. When such an operator is moreover positive or contractive, we show that the operator is induced by either a continuous group homomorphism or a continuous group anti-homomorphism.
If the groups are totally disconnected, bounded homomorphisms on the Fourier algebra can be realised by the idempotent preserving operators.
\end{abstract}

\maketitle

\section{Introduction}
Let $G$ be a locally compact Hausdorff group. The Fourier-Stieltjes $B(G)$ and the Fourier algebras $A(G)$ of $G$ were introduced by Eymard in his celebrating paper \cite{Eym}. Recall that $B(G)$ is the linear combination of all continuous positive definite functions on $G$, as a Banach space, $B(G)$ is naturally isometric to the predual of $W^*(G)$, the von Neumann algebras generated by the universal representations $\omega_G$ of $G$. Moreover, it is a commutative Banach $*$-algebra with respect to pointwise multiplication and complex conjugation. The Fourier algebra $A(G)$ is the closed ideal of $B(G)$ generated by the functions with compact supports. As a Banach space, $A(G)$ is isometric to the predual of the group von Neumann algebra $\vN(G)$, the von Neumann algebra generated by the left regular representations $\lambda_G$ of $G$.
It is well known that $A(G)$ is regular, semisimple, and the Fourier and the Fourier-Stieltjes algebras are both subalgebras of $C_b(G)$, the algebra of continuous bounded functions on $G$. 

Takesaki and Tatsuuma in \cite{TT} showed that there is a one-to-one correspondence between compact subgroups of $G$ and non-zero right invariant closed self-adjoint subalgebras of $A(G)$. As a refinement, Bekka, Lau and Schlichting in \cite{BLS} studied non-zero, closed, invariant $*$-subalgebras of $A(G)$. They showed that these spaces are the Fourier algebras $A(G/K)$ of the quotient group $G/K$ for some compact normal subgroup $K$ of $G$. On the other hand, Forrest \cite{Fo} introduced the Fourier algebra $A(G/K)$ of the left coset space $G/K$, where $K$ is a compact (not necessary to be normal) subgroup of the locally compact group $G$. This algebra can simultaneously be viewed as an algebra of functions on $G/K$ and as the subalgebra of $A(G)$ consisting of functions in $A(G)$ which are constants on left cosets of $K$. Note that $A(G/K)$ is regular and semisimple, the spectrum $\sigma(A(G/K))$ is $G/K$, and it is a norm closed left translation invariant $*$-subalgebra of $A(G)$.

A longstanding question in harmonic analysis is to determine all homomorphisms of Fourier or Fourier-Stieltjes algebras of any locally compact groups. For any pair of locally compact abelian groups $G$ and $H$, Cohen \cite{Co} characterised all bounded homomorphisms from the group algebra $L^1(G)$ to the measure algebra $M(H)$. In doing so, he made use of a profound discovery of his characterisation of idempotent measures on the groups. Cohen's results were generalized by Host in \cite{Host}, where he discovered the general form of idempotents in the Fourier-Stieltjes algebras, and characterised bounded homomorphisms from $A(G)$ to $B(H)$ when the group $G$ has an abelian subgroup of finite index. Further generalizations were made in \cite{Il, IlieSpronk} for any locally compact amenable group $G$, where completely bounded homomorphisms from $A(G)$ into $B(H)$ were characterised by continuous piecewise affine maps (see also \cite{Daws}). Most general results were given by Le Pham in \cite{LePham}, he determined all contractive homomorphisms from $A(G)$ into $B(H)$ for any locally compact groups $G$ and $H$. 

To describe idempotent elements in the Fourier-Stieltjes and the Fourier algebras, we first recall some terminologies. Let $G$ be a group and $K$ be a subgroup of $G$, we see that $Ks= ss^{-1}Ks$ for any $s\in G$, which means that we need not distinguish between left and right cosets of the group $G$. The \emph{coset ring} of $G$, denoted $\Omega(G)$, is the smallest ring of subsets of $G$ containing all cosets of subgroups of $G$. We denote $\Omo(G)$ the ring of subsets generated by open cosets of $G$, and similarly, $\Omoc(G)$ the ring of subsets generated by compact open cosets of $G$. By \cite{Host}, idempotents in the Fourier-Stieltjes algebra $B(G)$ are the indicator functions $1_F$ of an element $F$ of $\Omo(G)$. Let $I_B(G)$ be the set of all idempotent elements in $B(G)$. We denote the closure of the span of $I_B(G)$ by $B_I(G)$. From \cite[Proposition 1.1]{IS2}, we have that 
\begin{align*}
   A(G) \cap I_B(G)= \{1_Y: Y\in \Omoc(G)\},
\end{align*}
which gives rise idempotents in $A(G)$, denoted by $I(G)$. Let  $A_I(G)$ be the subalgebra of $A(G)$ generated by $I(G)$. Note that Ilie and Spronk \cite{IlieSpronk} showed that $1_F$ is an idempotent in $B(G)$ with $\|1_F\|_{B(G)}= 1$ if and only if $F$ is an open coset in $G$, however, there are idempotents with small norms \cite{MLePham} or with large norms \cite{Anoussis}. Moreover, the existence of idempotents of arbitrarily large norm implies the existence of homomorphisms of arbitrarily large norm (see \cite{Anoussis} for details). Thus, idempotent elements play an essential role in the study of homomorphisms on Fourier algebras. It is of its own interest to study the norms of idempotent elements in Fourier-Stieltjes and Fourier algebras, for our purpose we will focus solely on operators which preserve idempotents. 

In the rich literature of linear preservers, there are many works that study linear maps $T$ on spaces $X$ which preserve some subsets $S$ of $X$, i.e., $T(S) \subset S$.
Dieudonn\'e in \cite{Dieu} studied semi-linear maps on $M_n(\mathbb{K})$, the algebra of $n \times n$ matrices over a field $\mathbb{K}$, which preserve the set of all singular matrices. After that, many mathematicians considered linear maps on $M_n(\mathbb{K})$ that preserve subsets of matrices with different properties (e.g. \cite{Jacob, MM, Di, BPW} to name a few). 
In \cite{BS1}, it is shown that every complex linear map $T$ on $M_n(\mathbb{C})$ which preserves the set of all idempotents is either an inner automorphism or an inner anti-automorphism. In addition, in \cite{BS2} linear maps on $M_n(\mathbb{C})$ which send potent matrices (that is, matrices $A$ satisfy $A^r= A$ for some integer $r\geq 2$) to potent matrices were characterised. 
Since then, the studies of idempotent preserving maps have attracted considerable interest, see, e.g. \cite{Dolinar, GLS}. 
Recently, in \cite{LTWW} the authors proved that every additive map from the rational span of Hermitian idempotents in a von Neumann algebra into the rational span of Hermitian idempotents in a C*-algebra can be extended to a Jordan $*$-homomorphism. 

In this paper, we study bounded linear operators from $A(G)$ into $B(H)$ which send idempotents to idempotents. 
We show that such an operator will give rise an algebraic homomorphism on $A_I(G)$. The algebra $A_I(G)$ will be our main object of study, namely, we will characterise linear mappings defined on the Fourier algebra $A(G)$ or on $A_I(G)$ which preserve $I(G)$. Moreover, we show that when the groups are totally disconnected, idempotent preserving operators will recover algebraic homomorphisms on the Fourier algebra. 

\section{Main results}
Let $G$ be a locally compact Hausdorff group and $K$ be a closed subgroup of $G$. We will denote by $G/K$ the homogeneous space of left cosets of $K$. Let 
\begin{align*}
B(G: K):= \{u\in B(G): u(xk)= u(x) \, \text{for all } x\in G, k\in K\},
\end{align*}
this is, functions in $B(G)$ which are constant on cosets of $K$, and 
\begin{align*}
A(G: K):= \{u\in B(G: K): q(\supp(u)) \text{ is compact in } G/K\}^{-_{B(G)}},
\end{align*}
where $\supp(u)$ is the support of $u$ in $G$ and $q$ is the canonical quotient map from $G$ to $G/K$. If furthermore $K$ is a normal subgroup, by \cite[Proposition 3.2]{Fo} we have that $B(G: K)$ and $A(G: K)$ are isometrically isomorphic to the Fourier-Stieltjes and the Fourier algebras $B(G/K)$ and $A(G/K)$, respectively. Note that $A(G:K)\cap A(G)\neq \{0\}$ if and only if $K$ is compact.

Let $e$ be the identity of the group $G$, we denote the connected component of $e$ by $G_e$ which is a closed normal subgroup of $G$, thus, $G/G_e$ is a totally disconnected locally compact group. The following result about the algebra $A_I(G)$ generated by idempotents of $A(G)$ in relation with $A(G: G_e)$ was given in \cite{IS2}, for the completion we give a short proof in the paper.

\begin{proposition}\cite[Proposition 1.1(ii)]{IS2}\label{quotient}
If the connected component $G_e$ is compact then $A_I(G)= A(G:G_e)$, this is, $A_I(G)$ consists of all functions in $A(G)$ which are constant on cosets of $G_e$. In particular, if $G_e$ is not compact then $A_I(G)= \{0\}$. 
\end{proposition}

\begin{proof}
Let $q_G: G \to G/G_e$ is the quotient map onto $G/G_e$. Since $G_e$ is compact, via $u \mapsto u\circ q_G$ we have that $A(G/G_e)$ is isometrically isomorphic to $A(G:G_e)$, which is a closed subalgebra of $A(G)$. Thus, $A_I(G)= \overline{\spnn} \{1_Y: Y\in \Omoc(G)\} \subseteq A(G:G_e)$. Conversely, since $G/G_e$ is totally disconnected, we have that the span of the idempotents of $A(G/G_e)$ is dense \cite[Theorem 5.3]{Fo}. Moreover  $A(G/G_e)$ is isomorphic to $A(G:G_e)$, thus, $A(G:G_e)$ is generated by idempotents of $A(G/G_e)$, so $A(G:G_e) \subseteq A_I(G)$. 
\end{proof} 

If the Fourier algebra which contains non-trivial idempotents, that is, the connected component $G_e$ is compact, then by Proposition \ref{quotient} there is an isometric isomorphism from $A_I(G)$ onto $A(G/G_e)$. More precisely, it induces an isometric isomorphism $\varphi_G: A_I(G) \to A(G/G_e)$ as 
\begin{align}\label{eq:varphi_G}
    \varphi_G(f)(q_G(a))=f(a)
\end{align}
for any $f \in A_I(G)$ and $a \in G$, where $q_G: G \to G/G_e$ is the quotient map onto $G/G_e$.

\subsection{Idempotent preserving maps with $T(I(G)) \subset I_B(H)$}\label{sec:idem preserver}

Let $G$ and $H$ be two locally compact groups. We consider a bounded complex linear map  $T: A(G) \to B(H)$ which satisfies 
\begin{align}\label{eq:prop_T}
T(I(G)) \subset I_B(H).
\end{align}
For any $f \in \spnn \{1_Y: Y\in \Omoc(G)\}$, there exists $\alpha_i \in \mathbb{C}$ and $Y_i \in \Omoc(G)$ such that $f= \Sigma_{k= 1}^{n} \alpha_k 1_{Y_k}$. Thus, we have $Tf=\Sigma^{n}_{k=1} \alpha_k T1_{Y_k} \in \spnn \{1_Y: Y\in \Omo(H)\} \subset B(H)$. Let us recall that $A_I(G)= \overline{\spnn} \{1_Y: Y\in \Omoc(G)\}$ and $B_I(H)= \overline{\spnn} \{1_Y: Y\in \Omo(H)\}$. Since $T$ is a bounded map, we obtain $T(A_I(G)) \subseteq B_I(H) \subset B(H)$.

Our aim is to obtain a representation of such a map $T$ on $A_I(G)$.  
If $I(G)=\{0\}$, then $A_I(G)=\{0\}$. Since $T$ is complex linear, we have $T=0$ on $A_I(G)$. 
Thus without loss of generality, we can assume that the Fourier algebra $A(G)$ have non-zero idempotent elements. 
Hence, the connected component $G_e$ is always a compact normal subgroup of $G$. On the other hand, we define the following map which will be used in sequel. 

\begin{definition}
Let $G$ be a locally compact Hausdorff group. Using the 
axiom of choice, let $S$ be a set of 
representatives of the cosets of $\Ge$, that is $G=\bigsqcup_{a\in S}aG_e$. Then we define a map $[\ \cdot \ ]_{\Ge}$ from $\Ge$ onto $S$ by 
\[
[aG_e]_{\Ge}=a
\]
for any $a \in S$. 
\end{definition}

We first have the following observations concerning the operator satisfying \eqref{eq:prop_T}. 

\begin{lemma}\label{idempotentproduct}
 The map $T$ preserves the disjointness of idempotents. This is, $Tf \cdot Tg= 0$ for any $f,g \in I(G)$ with $f\cdot g= 0$. 
 \end{lemma}

\begin{proof}
Let $f, g \in I(G)$ such that $f\cdot g= 0$. Then we have $(f+g)^2=f+g$. Thus $f+g \in I(G)$. By the assumption, $Tf$, $Tg$ and $T(f+g) \in I_B(H)$. Since we have $(T(f+g))^2=Tf+Tg$, we get $Tf \cdot Tg=0$.
\end{proof}

\begin{definition}
We define $\Phi: \Ag \to B(H)$ by 
\[
\Phi (f)= T \circ \varphi_G^{-1}(f)
\]
for any $f \in \Ag$,
where $\varphi_G$ is given in \eqref{eq:varphi_G}.
\end{definition}

Then $\Phi$ is a bounded complex linear operator from $\Ag$ into $B(H)$. In order to achieve our main result, we consider the dual map $\Phi^{*}: W^*(H) \to \vN(G/G_e)$ and have the following lemmas.

\begin{lemma}\label{support}
Let $\la \in \vN(\Ge)$ and $a \in \Ge$. Suppose that $a \in \supp \la$. Then for every neighbourhood $V$ of $a$ in $\Ge$, there exists $h \in I(\Ge)$ such that $\supp h \subset V$ and $\langle \la, h \rangle \neq 0$.
\end{lemma}

\begin{proof}
Since $\Ge$ is totally disconnected, every neighbourhood of the identity contains an open compact subgroup. As $a^{-1}V$ is a neighbourhood of the identity, there exists an open compact subgroup $G_{a}$ in $\Ge$ such that $G_{a} \subset a^{-1}V$. Thus $aG_{a} \subset V$. Since $aG_{a}$ is a compact open coset in $\Ge$, we have that $1_{aG_{a}} \in \Ag$ is an idempotent with norm $1$. Since $a \in \supp \la$, there is $g \in \Ag$ such that $\supp g \subset aG_{a}$ and $\langle \la, g \rangle \neq 0$. Put $\delta= |\langle \la, g \rangle|$. As $\varphi_G^{-1}(g) \in A_{I}(G)$, there are $\alpha_i \in \C$ and $f_i \in I(G)$ such that $\|\varphi_G^{-1}(g)-\sum^{n}_{i=1}\alpha_if_i\| < \delta / \|\la\|$ for some $n\in \N$. Since $\varphi_G$ is an isometric isomorphism, we have $\|g-\sum^{n}_{i=1}\alpha_i\varphi_G(f_i)\| < \delta / \|\la\| $ and $\varphi_G(f_i) \in I(\Ge)$. Then we obtain 
\begin{align*}
1_{aG_{a}}(g-\sum^{n}_{i=1}\alpha_i\varphi_G(f_i))=g-\sum^{n}_{i=1}\alpha_i1_{aG_{a}}\varphi_G(f_i),
\end{align*}
thus,
\begin{align*}
\|g-\sum^{n}_{i=1}\alpha_i1_{a G_{a}}\varphi_G(f_i)\| \le  \|g-\sum^{n}_{i=1}\alpha_i\varphi_G(f_i)\| <\frac{\delta}{\|\la\|}.
\end{align*}

Suppose for every $1 \le i \le n$, we have $\langle \la, 1_{aG_{a}}\varphi_G(f_i) \rangle= 0$. Then 
\begin{equation*}
    \begin{split}
      |\langle \la, g \rangle| &= |\langle \la, g \rangle - \sum^{n}_{i=1}\alpha_i \langle \la, 1_{aG_{a}}\varphi_G(f_i) \rangle | \\
      &= |\langle \la, (g-\sum^{n}_{i=1}\alpha_i 1_{aG_{a}}\varphi_G(f_i) )\rangle | \\
      & \le \|\la\| \|g-\sum^{n}_{i=1}\alpha_i 1_{aG_{a}}\varphi_G(f_i) \| \\
      &< \|\la\| \frac{\delta}{\|\la\|}=\delta.
    \end{split}
\end{equation*}
This implies that $|\langle \la, g \rangle|< \delta$, which is a contradiction. Therefore, there is an $i_0 \in \{1, \cdots, n \}$ such that 
\[
\langle \la, 1_{aG_{a}}\varphi_G(f_{i_0}) \rangle \neq 0.
\]
We also have $\supp(1_{aG_{a}}\varphi_G(f_{i_0})) \subset V$ and $1_{aG_{a}}\varphi_G(f_{i_0}) \in I(\Ge)$, the proof is thus completed.
\end{proof}

\begin{proposition}\label{existance of psi}
For any $a \in H$, there exist uniquely $b \in G/G_e$ and $\alpha \in \C$ such that $\Phi^{*}(\omega_{H}(a))=\alpha \lambda_{\Ge}(b)$. 
\end{proposition}

\begin{proof}
Suppose there were $b_1, b_2 \in \Ge$ such that $b_1, b_2$ were both in $\supp(\Phi^{*}(\omega_{H}(a)))$. Since $G_e$ is a closed subgroup of $G$, the quotient group $\Ge$ is  Hausdorff. Thus, there are neighbourhoods $V_{b_1}$ and $V_{b_2}$ of $b_1$ and $b_2$, respectively, in $\Ge$ such that $V_{b_1} \cap V_{b_2}= \emptyset$. By Lemma \ref{support}, there are $h_i \in I(\Ge)$, for $i=1,2$, such that $\supp h_i \subset V_{b_i}$ and $\langle \Phi^{*}(\omega_{H}(a)), h_i \rangle \neq 0$. As $V_{b_1} \cap V_{b_2}= \emptyset$, we get $h_1h_2=0$. Since $\varphi_G$ is an isomorphism, we have $\varphi_G^{-1}(h_i) \in I(G)$, for $i= 1, 2$, and $\varphi_G^{-1}(h_1)\cdot \varphi_G^{-1}(h_2)=\varphi_G^{-1}(h_1 h_2)=0$. By Lemma \ref{idempotentproduct}, we have $T(\varphi_G^{-1}(h_1))\cdot T(\varphi_G^{-1}(h_2)) = 0$. On the other hand, we obtain
\begin{multline*}
 0 \neq \langle \Phi^{*}(\omega_{H}(a)), h_1 \rangle=\Phi(h_1)(a)= T \circ \varphi_G^{-1}(h_1)(a)= T(\varphi_G^{-1}(h_1))(a),   
\end{multline*}
and 
\begin{multline*}
0 \neq \langle \Phi^{*}(\omega_{H}(a)), h_2 \rangle=\Phi(h_2)(a)=T \circ \varphi_G^{-1}(h_2)(a)= T(\varphi_G^{-1}(h_2))(a).
\end{multline*}
Therefore, $$T(\varphi_G^{-1}(h_1))\cdot T(\varphi_G^{-1}(h_2))\neq 0,$$ this is a contradiction. Since $\supp (\Phi^{*}(\omega_{H}(a))) \neq \emptyset$, there is uniquely $b \in \Ge$ such that $\supp (\Phi^{*}(\omega_{H}(a)))=\{b\}$. Consequently, by \cite[Corollary 2.5.9]{KL}, there is an $\alpha \in \C$ such that $\Phi^{*}(\omega_{H}(a))=\alpha \lambda_{\Ge}(b)$.
\end{proof}

For any $a \in H$, by Proposition \ref{existance of psi}, there are unique $b \in G/G_e$ and $\alpha \in \C$ such that $\Phi^{*}(\omega_{H}(a))=\alpha \lambda_{\Ge}(b)$, thus we have 
\[
\Phi(f)(a)=\alpha f(b),
\]
for any $f \in \Ag$. We define $\phi: H \to \C$ by $\alpha=\phi(a)$. We also define $\psi: H \to \Ge$ by $b=\psi(a)$. Then we get 
\begin{equation}\label{formofPhi}
      \Phi(f)(a)=\phi(a)f(\psi(a)),
\end{equation}
for any $f \in \Ag$ and $a \in H$.

For any $h \in I(G)$, since we have $\Phi(\varphi_G(h))=T(h)\in I_B(H)$, we obtain that 
\begin{equation*}
    \begin{split}
        (\Phi(\varphi_G(h)))^2&=T(h)T(h)=T(h)=\Phi(\varphi_G(h)).\\
    \end{split}
\end{equation*}
On the other hand, since $(\varphi_G(h))^2=\varphi_G(h^2)=\varphi_G(h)$ in $A(\Ge)$, we obtain that
\begin{equation}\label{idempotentcondition}
    \phi(a)^2\varphi_G(h)(\psi(a))=\phi(a)\varphi_G(h)(\psi(a))
\end{equation}
for any $h \in I(G)$ and $a \in H$.
\begin{lemma}\label{existenceofidempotent}
For any $a \in H$, there is an idempotent $1_{\psi(a)G_0}$ of $A(\Ge)$ where $\psi(a)G_0$ is an open compact neighbourhood of $\psi(a)$.
\end{lemma}
\begin{proof}
As $\Ge$ is totally disconnected, there is an open compact subgroup $G_0$ in $\Ge$.  For any $\psi(a) \in \Ge$, $\psi(a)G_0$ is a compact open coset in $\Ge$, hence $1_{\psi(a)G_0}$ is an idempotent of $A(\Ge)$ with norm $1$. 
\end{proof}

\begin{lemma}\label{algebrahomo}
The map $\Phi: A(\Ge) \to B(H)$ is an algebraic homomorphism. 
\end{lemma}
\begin{proof}
Let $a \in H$. By Lemma \ref{existenceofidempotent}, there is an idempotent $1_{\psi(a)G_0}$ of $A(\Ge)$.
Since $\varphi_G: A_I(G) \to A(\Ge)$ is surjective, there is $f \in A_I(G)$ such that $\varphi_G(f)=1_{\psi(a)G_0}$. Moreover, we have that $f^2=(\varphi_G^{-1}(1_{\psi(a)G_0}))^2=\varphi_G^{-1}(1_{\psi(a)G_0})=f$, this implies that $f \in I(G)$. Thus by (\ref{idempotentcondition}), we have 
\begin{multline*}
    \phi(a)^2=\phi(a)^2 1_{\psi(a)G_0}(\psi(a))=\phi(a)^2\varphi_G(f)(\psi(a))\\=\phi(a)\varphi_G(f)(\psi(a))=\phi(a)1_{\psi(a)G_0}(\psi(a))=\phi(a).
\end{multline*}
Since $a \in H$ is arbitrary, we have
\begin{equation}\label{square}
    \phi^2=\phi
\end{equation}
on $H$, thus we get $\phi: H \to \{0,1\}$. 
In addition, for any $f, g \in A(\Ge)$ and $a \in H$, we have
\begin{multline*}
    \Phi(fg)(a)=\phi(a)(fg)(\psi(a))=\phi(a)^2(fg)(\psi(a))\\
    =\phi(a)f(\psi(a))\phi(a)g(\psi(a))=(\Phi(f)\Phi(g))(a).
\end{multline*}
Hence $\Phi$ is an algebraic homomorphism from $A(\Ge)$ into $B(H)$. 
\end{proof}

\begin{lemma}\label{psicontinuous}
The map $\psi:\phi^{-1}(1) \to \Ge$ is continuous. 
\end{lemma}
\begin{proof}
For any $a_0 \in \phi^{-1}(1) \subset H$, let $U$ be an open neighbourhood of $\psi(a_0)$ in $\Ge$. Then there is $f_0 \in \Ag$ such that
\[
f_0(\psi(a_0))=1 \quad \text{and} \quad  f_0(b)=0 \, \, \text{ for } b \in (\Ge) \setminus U.
\]
Let $(a_\lambda)_\lambda \subseteq \phi^{-1}(1)$ be a net such that $a_{\lambda} \to a_0$. As $\Phi(f_0) \in B(H)$, $\Phi f_0(a_{\lambda}) \to \Phi f_0(a_{0})=f_0(\psi(a_0))=1$. There is an $\lambda_0 $ such that if $\lambda \ge \lambda_0$ then $|\Phi f_0(a_{\lambda})|>\frac{1}{2}$. Since $\Phi f_0(a_{\lambda})= f_0(\psi(a_{\lambda}))$, we have $\psi(a_{\lambda}) \in U$ provided $\lambda \ge \lambda_0$. Thus $\psi$ is continuous on $\phi^{-1}(1)$. 
\end{proof}

\begin{lemma}\label{phiopen}
The set $\phi^{-1}(1)$ is an open subset of $H$. 
\end{lemma}
\begin{proof}
Let $a \in \phi^{-1}(1)$ be arbitrary. By Lemma \ref{existenceofidempotent}, there is an idempotent $1_{\psi(a)G_0}$ of $A(\Ge)$ where $\psi(a)G_0$ is an open compact neighbourhood of $\psi(a)$. Since $\Phi(1_{\psi(a)G_0}) \in B(H) \subset C_b(H)$, there exists an open neighbourhood $V$ of $a$ in $H$ such that if $b \in V$ then 
\[
|\Phi(1_{\psi(a)G_0})(a)-\Phi(1_{\psi(a)G_0})(b)|\le \frac{1}{2}.
\]
We have 
\[
|1-\phi(b)1_{\psi(a)G_0}(\psi(b))|=|\Phi(1_{\psi(a)G_0})(a)-\Phi(1_{\psi(a)G_0})(b)| \le \frac{1}{2}.
\]
Since either $\phi(b)1_{\psi(a)G_0}(\psi(b))= 1$ or $\phi(b)1_{\psi(a)G_0}(\psi(b))= 0$, this implies that 
\[
\phi(b)1_{\psi(a)G_0}(\psi(b))=1.
\]
Hence we have $\phi(b)=1$ for any $b \in V$. Thus $V \subset \phi^{-1}(1)$. It follows that $\phi^{-1}(1)$ is an open subset of $H$.
\end{proof}

\begin{theorem}\label{A(G)toB(H)}
Let $G$ and $H$ be two locally compact Hausdorff groups, and $T:A(G) \to B(H)$ be a bounded complex linear operator. Suppose $T$ satisfies that $T(I(G)) \subset I_B(H)$. Then 
there are an open subset $U$ of $H$ and a continuous map $\psi$ from $U$ into $\Ge$ such that 
\begin{equation}\label{form of T}
   Tf(a)=
\begin{cases}
     f([\psi(a)]_{G_e}) \quad & \text{if $a \in U$,} \\
    0                 & \text{if $ a \in H \setminus U$,}
  \end{cases}
\end{equation}
for any $f \in A_I(G)$ and $a \in H$. 
\end{theorem}
\begin{proof}
Let $U=\phi^{-1}(1)$. By Lemma \ref{phiopen}, $U$ is an open subset of $H$. Moreover, Lemma \ref{psicontinuous} shows that $\psi: U \to \Ge$ is a continuous map. 
Applying (\ref{formofPhi}), for any $f \in A_I(G)$ and $a\in H$, we have 
\[
Tf(a)=\Phi(\varphi_G(f))(a)=\phi(a)\varphi_G(f)(\psi(a)).
\]
Thus, we get 
 \[
  \begin{split}
Tf(a)&=
\begin{cases}
    \varphi_G(f)(\psi(a)) \quad & \text{if $a \in U$,} \\
    0                 & \text{if $ a \in H \setminus U$}
  \end{cases}\\
  &=
  \begin{cases}
    f([\psi(a)]_{G_e}) \quad & \text{if $a \in U$,} \\
    0                 & \text{if $ a \in H \setminus U$,}
  \end{cases}\\
  \end{split}
  \]
  for any $f \in A_I(G)$ and any $a \in H$. 
\end{proof}

The following example shows that the assumption in Theorem \ref{A(G)toB(H)} does not imply  $T(I(G))\subset I(H)$. This observation is in line with the well known fact that $f\circ \psi$ may not be in the Fourier algebra $A(H)$ in general (see Remark \ref{re:not in AH}). 
\begin{example}
Let $G=\{0\}$ be the trivial group. Then we define a bounded linear operator $T:A(G) \to B(\mathbb{Z})$ by $T(1_{G})=1_\mathbb{Z}$. Then it satisfies $T(I(G)) \subset I_B(H)$. Note that in this case, we have $U=\mathbb{Z}$ and the continuous map $\psi: \mathbb{Z} \to \Ge$ is $\psi(n)=0$ for any $n \in \mathbb{Z}$. On the other hand, since $1_\mathbb{Z} \notin A(\mathbb{Z})$, we have $T(I(G))\nsubseteq I(H)$. 
\end{example}

\begin{remark}\label{re:not in AH}
In general the converse statement of above theorem may not hold since we do not know if $Tf \in A(H)$ for any $f \in A_I(G)$, even $T$ has a representation of the form (\ref{form of T}). If we only have $\psi: U \subseteq H \to G$ being continuous, then $f\mapsto f\circ \psi$ maps $A(G)$ into $\ell^{\infty}(H)$ in general. 
For abelian groups $G$ and $H$, Cohen \cite{Co} showed that $f\mapsto f\circ \psi$ maps $A(G)$ to $B(H)$ if and only if $\psi$ is a continuous piecewise affine map from a set in the open coset ring of $H$ into $G$. 
This characterisation was extended by Host \cite{Host} to the case when $G$ has an abelian subgroup of finite index and $H$ is arbitrary, and by \cite{LePham} to general groups. 
\end{remark}

Under extra assumptions on $T$, we obtain  algebraic structures for the open set $U$ and  algebraic properties on the map $\psi$. Let us first recall positive operators on the Fourier algebra.

A bounded linear operator $T:A_I(G) \to B(H)$ is said to be positive if $T(u)$ is positive definite whenever $u \in A_I(G)$ is a positive definite function.

\begin{corollary}
Let $G$ and $H$ be two locally compact groups. Let $T:A_I(G) \to B(H)$ be a positive bounded complex linear operator. If $T$ satisfies that $T(I(G)) \subset I_B(H)$ then there exists an open subgroup $U$ of $H$ and a continuous group homomorphism or anti-homomorphism $\psi$ from the open subgroup $U$ of $H$ into $\Ge$ such that 
\[
Tf(a)=
\begin{cases}
     f([\psi(a)]_{G_e}) \quad & \text{if $a \in U$,}\\
    0                 & \text{if $ a \in H \setminus U$,}
  \end{cases}
  \]
  for any $f \in A_I(G)$ and $a \in G$.
\end{corollary}
\begin{proof}
Since isometric isomorphism $\varphi_G$ preserves positivity, $u \in A_I(G)$ is a positive definite function if and only if $\varphi_G(u)$ is positive definite. This implies that $T$ is positive if and only if $\Phi$ is positive, thus, $\Phi:A(\Ge) \to B(H)$ is a positive homomorphism by Lemma \ref{algebrahomo}. It follows from \cite[Theorem 4.3]{LePham} that there exists an open subgroup $U$ of $H$ and a continuous group homomorphism or anti-homomorphism $\psi$ from $U$ into $\Ge$ such that for any $f \in A(\Ge)$, 
$\Phi f$ is either equal to $f\circ \psi$ in $U$, or $0$ otherwise. 
Thus, we have
\[
Tf(a)=
\begin{cases}
     f([\psi(a)]_{G_e}) \quad & \text{if $a \in U$,}\\
    0                 & \text{if $ a \in H \setminus U$,}
  \end{cases}
  \]
  for any $f \in A_I(G)$ and $a \in H$.
\end{proof}

\begin{corollary}
Let $G$ and $H$ be two locally compact groups, and $T:A_{I}(G) \to B(H)$ be a contractive complex linear operator. If $T$ satisfies that $T(I(G)) \subset I_B(H)$ then
there exists an open subgroup $U$ of $H$, a continuous group homomorphism or anti-homomorphism $\psi$ from $U$ into $\Ge$, and elements $b \in G$ and $c \in H$ such that 
\[
Tf(a)=
\begin{cases}
    f(b[\psi(ca)]_{G_e}) \quad & \text{if $a \in c^{-1}U$,}\\
    0                 & \text{if $a \in H \setminus c^{-1}U$.}
  \end{cases}
  \]
\end{corollary}

\begin{proof}
Since $\varphi_G$ is an isometric isomorphism, if $T$ is contractive then $\Phi$ is also a contractive operator. By Lemma \ref{algebrahomo}, $\Phi$ is a contractive homomorphism from $A(\Ge)$ into $B(H)$. It follows from \cite[Theorem 5.1]{LePham} that there exists an open subgroup $U$ of $H$, a continuous group homomorphism or anti-homomorphism $\psi$ from $U$ into $\Ge$, and elements $bG_e \in \Ge$ and $c \in H$ such that for any $f \in A(\Ge)$ and $a \in H$, $\Phi f(a)= f(bG_e\psi(ca))$ provided $a \in c^{-1}U$, otherwise, $\Phi f(a)= 0$. By recalling the definition of $\Phi$, we have the characterisation of $T$. 
\end{proof}

\subsection{Idempotent preserving maps with $T(I(G)) \subset I(H)$}

Let us assume that the bounded linear operator $T:A(G) \to B(H)$ satisfies $T(I(G)) \subset I(H)$. Then naturally we obtain $T(A_I(G)) \subseteq A_I(H)$.

We define $T_q: \Ag \to \Ah$ by 
\[
T_q(f)=\varphi_H \circ T \circ \varphi_G^{-1}(f)=\varphi_H \circ \Phi(f),
\]
for any $f \in \Ag$, where $\varphi_H: A_I(H) \to \Ah$ is an isometric isomorphism defined similarly as in \eqref{eq:varphi_G}. Note that $T_q$ is an algebraic homomorphism.

\begin{lemma}\label{equivalentclass}
Let $a \in \phi^{-1}(1) \subset H$ and $b\in H$ such that $a^{-1}b \in H_e$. Then $\phi(b)=1$ and $\psi(a)=\psi(b)$.
\end{lemma}
\begin{proof}
Suppose that $\psi(a)\neq \psi(b)$. By (\ref{formofPhi}), we have $\Phi: \Ag \to B(H)$ such that for any $f \in \Ag$,
\[
\Phi(f)(a)=f(\psi(a))
\]
and 
\[
\Phi(f)(b)=\phi(b)f(\psi(b)).
\]
Since $\Ge$ is Hausdorff, there are disjoint open neighbourhoods $V_a$ and $V_b$ of $\psi(a)$ and $\psi(b)$, respectively, in $\Ge$. By Lemma \ref{support}, for $\lambda_{\Ge}(\psi(a)) \in VN(\Ge)$, there is $h \in I(\Ge)$ such that $\supp h \subset V_a$ and $h(\psi(a))\neq 0$. Since $a \in \phi^{-1}(1)$, we get 
\[
\Phi(h)(a)=h(\psi(a))\neq 0
\]
and 
\[
\Phi(h)(b)=\phi(b)h(\psi(b))=0.
\]
By the assumption that $T(I(G)) \subset I(H)$ and $\varphi_G^{-1}(h)\in I(G)$, we have $\Phi(h)=T(\varphi_G^{-1}(h)) \in I(H)$, an idempotent in $A(H)$. Hence, there is $Y \in \Omoc(H)$ such that $1_Y=\Phi(h)$. Since $1_Y(a)=\Phi(h)(a)=h(\psi(a))\neq 0$, we have $a \in Y$.
In addition, $Y$ is a clopen subset of $H$ and $H_e$ is a connected component containing $e$, thus $aH_e \subset Y$. This implies that $b=aa^{-1}b \in Y$.
It follows that 
\[
1=1_Y(b)=\Phi(h)(b)=0.
\]
This is a contradiction. Thus we have $\psi(a)=\psi(b)$. Furthermore, suppose that $\phi(b)=0$. There is an $h \in I(\Ge)$ such that $h(\psi(b))\neq 0$. Thus there is $Y \in \Omoc(H)$ such that $1_Y=\Phi(h)$. By a similar argument, we have $1_Y(a)=\Phi(h)(a)=h(\psi(a))\neq 0$, $a \in Y$ and $b \in Y$.  We obtain that 
\[
1=1_Y(b)=\Phi(h)(b)=\phi(b)h(\psi(b))=0.
\]
This is a contradiction. Therefore, $\phi(b)=1$ and $\psi(a)= \psi(b)$.
\end{proof}

For any $a,b \in H$, the condition $a^{-1}b \in H_e$ induces an equivalence relation on $H$. 
Lemma \ref{equivalentclass} shows that $\phi: H \to \{0,1\}$ and $\psi:H \to \Ge$ are constant functions on each equivalence class. Thus these induce maps $\phi':\He \to \{0,1\}$ and $\psi': \phi'^{-1}(1) \to \Ge$ by
\[
\phi'(aH_e)=\phi(a) \quad \text{for any $a \in H$}
\]
and 
\[
\psi'(aH_e)=\psi(a) \quad \text{for any $aH_e \in \phi'^{-1}(1)$}.
\]
By Lemma \ref{psicontinuous}, the map $\psi:\phi^{-1}(1) \to \Ge$ is continuous. As we have $\phi'^{-1}(1)=q_{H}(\phi^{-1}(1))$, we obtain that $\psi': \phi'^{-1}(1) \to \Ge$ is continuous. 

\begin{theorem}\label{boundedmap}
Let $G$ and $H$ be two locally compact Hausdorff groups and $T:A(G) \to B(H)$ be a bounded complex linear operator. Suppose that $T$ satisfies $T(I(G)) \subset I(H)$.
Then there exists an open subset $U$ of $H$ and a continuous map $\psi'$ from an open subset $q_{H}(U)$ of $\He$ into $\Ge$ such that 
\begin{equation}\label{the form of T}
   Tf(a)=
\begin{cases}
     f([\psi'(aH_e)]_{G_e}) \quad & \text{if $a \in U$,} \\
    0                 & \text{if $ a \in H \setminus U$,}
  \end{cases}
\end{equation}
for any $f \in A_I(G)$.
\end{theorem}

\begin{proof}
Define $U=\phi^{-1}(1)$. Since $q_{H}: H \to \He$ is the quotient map and $U$ is a open subset of $H$ by Lemma \ref{phiopen}.  
By (\ref{formofPhi}), for any $f \in \Ag$ and $a \in H$, we have
\begin{multline}\label{Gammaform}
    T_q(f)(aH_e)=\varphi_H\circ \Phi(f)(aH_e)=\Phi(f)(a)
    =\phi'(aH_e)f(\psi(a)).
\end{multline}
We shall show that $\phi'^{-1}(1)$ is an open subset of $\He$.
Let $a \in \phi'^{-1}(1)$. By Lemma \ref{existenceofidempotent}, there is an idempotent $1_{\psi'(a)G_0}$ of $A(\Ge)$ where $\psi'(a)G_0$ is an open compact neighbourhood of $\psi'(a)$. Since $T_q(1_{\psi'(a)G_0}) \in A(\He) \subset C_0(\He)$, the space of all continuous functions on $\He$ vanishing at infinity, there exists an open neighbourhood $V$ of $a$ in $\He$ such that if $b \in q_H^{-1}(V)$ then 
\[
|T_q(1_{\psi'(a)G_0})(a)- T_q(1_{\psi'(a)G_0})(bH_e)|\le \frac{1}{2}.
\]
We have 
\[
|1-\phi'(bH_e)1_{\psi'(a)G_0}(\psi(b))|= |T_q(1_{\psi'(a)G_0})(a)- T_q(1_{\psi'(a)G_0})(bH_e)| \le \frac{1}{2}.
\]
Since either $\phi'(bH_e)1_{\psi'(a)G_0}(\psi(b))= 1$ or $\phi'(bH_e)1_{\psi'(a)G_0}(\psi(b))= 0$, this implies that 
\[
\phi'(bH_e)1_{\psi'(a)G_0}(\psi(b))=1.
\]
Hence we have $\phi'(bH_e)=1$ for any $b \in q_H^{-1}(V)$. Thus $V \subset \phi'^{-1}(1)$. It follows that $\phi'^{-1}(1)$ is an open subset of $\He$.
Let us recall that $\psi':q_{H}(U) \to \Ge $ is a continuous map. 
Applying (\ref{Gammaform}), we have 
\begin{equation*}
    \begin{split}
    T_q(f)(aH_e)&=\phi'(aH_e)f(\psi(a))\\
  &=
    \begin{cases}
     f(\psi'(aH_e)) \quad & \text{if $a \in U$,} \\
    0                 & \text{if $ a \in H \setminus U$,}
  \end{cases}\\
    \end{split}
\end{equation*}
for any $f \in A(\Ge)$ and $a \in H$. 
As we have
\begin{multline*}
    T_q(\varphi_G(f))(aH_e)=\varphi_H\circ T\circ \varphi_G^{-1}(\varphi_G(f))(aH_e)=\varphi_H \circ T(f)(aH_e)=Tf(a)
\end{multline*}
for any $f \in A_I(G)$ and $a \in H$, we get 
 \[
  \begin{split}
Tf(a)&=
\begin{cases}
    \varphi_G(f)(\psi'(aH_e)) \quad & \text{if $a \in U$,} \\
    0                 & \text{if $a \in H \setminus U$.}
  \end{cases}\\
  \end{split}
  \]
\end{proof}

\section{Idempotent preserving bijections on $A_I(G)$}
In this section we assume furthermore that the bounded linear operator $T: A(G) \to B(H)$ satisfies that $T(I(G)) \subset I(H)$ and $T|_{A_I(G)}$ is a bijection onto $A_I(H)$.

\begin{theorem}\label{Tisbijective}
Let $G$ and $H$ be two locally compact groups, and $T:A(G) \to B(H)$ be a bounded complex linear operator. Suppose the operator $T$ satisfies that $T(I(G)) \subset I(H)$ and $T|_{A_I(G)}:A_I(G) \to A_I(H)$ is bijective. Then 
there exists a homeomorphism $\psi: \He \to \Ge$ such that 
\[
Tf(a)=f([\psi(aH_e)]_{G_e})
\]
for all $f \in A_{I}(G)$ and $a \in H$.
\end{theorem}

\begin{proof}
Since $T|_{A_I(G)}$ is a bijective linear map, $\varphi_G$ and $\varphi_H$ are isometric isomorphisms, by the proof of Proposition \ref{quotient} we have $T_q:= \varphi_H \circ T|_{A_I(G)} \circ \varphi_G^{-1}$ is an isomorphism from $A(\Ge)$ onto $A(\He)$. 

Applying Theorem \ref{boundedmap}, there is an open subset $U$ of $H$ such that and a continuous map $\psi$ from an open subset $q_{H}(U)$ of $\He$ into $\Ge$ such that 
\begin{equation*}
   T_q(f)(a)=
\begin{cases}
     f(\psi(a)) \quad & \text{if $a \in q_{H}(U)$,} \\
    0                 & \text{if $ a \in (\He) \setminus q_{H}(U)$,}
  \end{cases}
\end{equation*}
for any $f \in A(\Ge)$. Since $T_q:A(\Ge) \to A(\He)$ is surjective and the Fourier algebra $A(\He)$ separates the points in $\He$, we have $q_{H}(U)=\He$. Thus $U=H$ and we have
\[
 T_q(f)(a)=f(\psi(a))
\]
for every $f \in A(\Ge)$ and  $a \in \He$.
For any $h \in I(H)$, there exists $h_q \in A(\He)$ with $h_q^2= h_q$ such that 
\[
\varphi_H(h)= h_q.
\]
Since $T_q$ is bijective, there exists $f_q \in A(\Ge)$ such that 
\[
T_q(f_q)= h_q.
\]
Moreover, since $T_q$ is an algebraic homomorphim, we have $T_q(f_q^2)=(T_q(f_q))^2= h_q^{2}= h_q= T_q(f_q)$. By the injectivity of $T_q$, we get $f_q^2= f_q$. On the other hand, as $\varphi_G$ is an isometric isomorphism from $A_I(G)$ onto $A(\Ge)$, there exists $f \in I(G)$ such that 
\[
\varphi_G(f)= f_q.
\]
Hence, we have
\[
T(f)=(\varphi_H^{-1}\circ T_q \circ \varphi_G)(f)=\varphi_H^{-1}\circ T_q(f_q)= \varphi_H^{-1}(h_q)= h.
\]
This implies that $T(I(G))= I(H)$. In particular, we have $T^{-1}(I(H)) \subset I(G)$.
Thus we can apply similar arguments to $T|_{A_I(G)}^{-1}: A_I(H) \to A_I(G)$ and to $T_q^{-1}=\varphi_G \circ T|_{A_I(G)}^{-1} \circ \varphi_H^{-1}$ on $A(\He)$, we can then define a continuous map $\tilde{\psi}: \Ge \to \He$ such that
\[
T_q^{-1}(g)(b)=g(\tilde{\psi}(b))
\]
for any $g \in \Ah$ and $b \in \Ge$.

For any $g\in \Ah$ and $a \in \He$, we have 
\[
g(a)=T_q(T_q^{-1}g)(a)=g(\tilde{\psi}(\psi(a))).
\]
Since the Fourier algebra $\Ah$ separates  points in $\He$, we get 
\begin{equation}\label{sH}
    a=\tilde{\psi}(\psi(a)) \quad \text{for} \quad a \in \He.
\end{equation}
Moreover, we obtain 
\[
f(b)=T_q^{-1}(T_q f)(b)=f(\psi(\tilde{\psi}(b))),
\]
for any $f \in \Ag$ and $b \in \Ge$.
Similarly, as $\Ag$ separates points in $\Ge$, we have 
\begin{equation}\label{sG}
    b=\psi(\tilde{\psi}(b)) \quad \text{for} \quad b \in \Ge.
\end{equation}
By (\ref{sH}) and (\ref{sG}), we have that $\psi:\He \to \Ge$ is a bijection and $\tilde{\psi}=\psi^{-1}$. Let us recall that $\psi$ and $\tilde{\psi}$ are continuous on $\He$ and $\Ge$, respectively. As $\tilde{\psi}=\psi^{-1}$, we have that $\psi$ is a homeomorphism. In addition, we obtain
\begin{equation*}
      T_q(f)(a)=f(\psi(a)) \quad \text{for }\, f \in \Ag, \, a \in \He.
\end{equation*}
Since $T=\varphi_H^{-1} \circ T_q \circ \varphi_G$, we get 
\[
Tf(a)=f([\psi(aH_e)]_{G_e})
\]
for all $f \in A_{I}(G)$ and $a \in H$.
\end{proof}

Note that the assumption of bijectivity in above theorem is needed for the function $\psi: \He \to \Ge$ to be a homeomorphism. 

\begin{example}
Let $G=\{1,2\}$ be a multiplicative group equipped with the discrete topology. Let $H=\{0\}$ be the trivial group.
We define $T: A(G) \to A(H)$ by $Tf(0)=f(1)$ for any $f \in A(G)$. Then $T$ is a bounded complex linear operator on $A(G)$ and for any $1_Y \in A(G)$, $T(1_Y)=1_H$ if $1 \in Y$, otherwise $T(1_Y)=0$. Thus $T(I(G))=I(H)$. On the other hand, $T(1_{\{1\}})=1_H=T(1_G)$, this implies that $T|_{A_I(G)}:A_I(G) \to A_I(H)$ is not injective. In addition, $\psi: \He= H \to \Ge= G$ satisfying 
\[
\psi(0)=1.
\]
is not a homeomorphism.
\end{example}

With extra assumptions on $T$ as in Section \ref{sec:idem preserver}, we obtain a characterisation of linear idempotent preserving maps between two Fourier algebras. Note that since the continuous map $\psi$ in the following two corollaries is either a group isomorphism or an anti-isomorphism, we naturally have $f\circ \psi\in A_I(H)$ for any $f \in A_I(G)$ (see \cite{Wal}), thus, we obtain a necessary and sufficient condition for the idempotent preserving operator $T$ on $A_I(G)$.

\begin{corollary}
Let $G$ and $H$ be two locally compact groups. A surjective complex linear contraction $T: A_I(G) \to A_I(H)$ satisfies  $T(I(G)) \subset I(H)$ if and only if
there exists a continuous group isomorphism or anti-isomorphism $\psi: \He \to \Ge$ and an element $b \in G$ such that 
\[
Tf(a)=f(b[\psi(aH_e)]_{G_e})
\]
for all $f \in A_{I}(G)$ and $a \in H$.
\end{corollary}

\begin{corollary}
Let $G$ and $H$ be two locally compact groups. A positive bounded complex linear bijection $T: A_I(G) \to A_I(H)$ satisfies $T(I(G)) \subset I(H)$ if and only if there exists a continuous group isomorphism or anti-isomorphism $\psi: \He \to \Ge$ such that 
\[
Tf(a)=f([\psi(aH_e)]_{G_e})
\]
for all $f \in A_{I}(G)$ and $a \in H$.

\end{corollary}

We will end our paper with a special case when the groups are totally disconnected. In such case, $A_I(G)$ is isometrically algebraic isomorphic to $A(G)$. Thus the \emph{idempotent preserving} operators recover the results of algebraic homomorphisms.

\begin{remark}
Suppose that $G$ and $H$ are totally disconnected locally compact groups. Let $T: A(G) \to A(H)$ be a bounded complex linear operator satisfying $T(I(G))\subset I(H)$. Then there exists a continuous map $\psi$ from an open subset $U$ of $H$ into $G$ such that 
\[
Tf(a)=
 \begin{cases}
    f(\psi(a)) \quad & \text{if $a \in U$,} \\
    0                 & \text{if $a \in H \setminus U$,}
  \end{cases}
\]
for any $f \in A(G)$ and $a \in H$. In addition, if $T$ is a surjective contraction or $T$ is a positive bijection, then it is equivalent to  $Tf = f \circ (b \psi)$ for some $b \in G$ or $Tf= f \circ \psi$, respectively,  for all $f\in A(G)$ where $\psi: H\to G$ is a continuous group isomorphism or group anti-isomorphism; in particular, $T$ is an algebraic homomorphism.

\end{remark}

\subsection*{Acknowledgments}
The second author was supported by JSPS KAKENHI Grant Numbers JP21K13804.

\end{document}